\documentclass[a4paper,reqno]{amsart}
\usepackage{amssymb, amsmath, amscd}
\theoremstyle{plain}
 \newtheorem{thm}{Theorem}[section]
 
 \newtheorem{lem}{Lemma}[section]
 
\theoremstyle{definition}

\theoremstyle{remark}
  
 \numberwithin{equation}{section}

\renewcommand{\ge}{\geqslant}

\title[Running title]{4-dimensional (para)-K\"ahler--Weyl structures}


\author[Gilkey]{\bfseries Peter Gilkey}
\address{
Mathematics Department, University of Oregon\\ Eugene OR 97403 USA}
\email{gilkey@uoregon.edu}
\author[Nik\v cevi\'c]{\bfseries Stana Nik\v cevi\'c}
\address{Mathematical Institute, Sanu, Knez Mihailova 36, p.p. 367\\11001 Beograd,
 Serbia}
  \email{stanan@mi.sanu.ac.rs}

\thanks{Research of the authors partially supported by project MTM2009-07756 
(Spain), and by project 174012 (Serbia)} 

\def\dvol{\mu}
\begin{document}

\vspace{18mm}
\setcounter{page}{1}
\thispagestyle{empty}

\begin{abstract}
We give an elementary proof of the fact that
any 4-dimensional para-Hermitian manifold admits a unique para-K\"ahler--Weyl structure.
 We then use analytic continuation to pass from the para-complex
 to the complex setting and thereby show any 4-dimensional pseudo-Hermitian manifold 
 also admits a unique K\"ahler--Weyl structure.
\end{abstract}

\maketitle

\section{Introduction}

\subsection{Weyl manifolds}
Let $(M,g)$ be a pseudo-Riemannian manifold of dimension $m$. A triple $(M,g,\nabla)$ is said to be a {\it Weyl manifold} and $\nabla$ is said to be a
{\it Weyl connection} if $\nabla$ is a torsion free connection with
$\nabla g=-2\phi\otimes g$ for some smooth $1$-form $\phi$. This is a conformal theory; if $\tilde g=e^{2f}g$ is a conformally equivalent metric,
then $(M,\tilde g,\nabla)$ is a Weyl manifold with associated $1$-form $\tilde\phi=\phi-df$.
If $\nabla^g$ is the Levi-Civita connection, we may then express $\nabla=\nabla^\phi$ in the form:
\begin{equation}\label{eqn-1.1}
\nabla^\phi_xy=\nabla_x^gy+\phi(x)y+\phi(y)x-g(x,y)\phi^{\#}
\end{equation}
where $\phi^{\#}$ is the dual vector field. Thus $\phi$ determines $\nabla$. Conversely, if $\phi$ is given and if we use Equation~(\ref{eqn-1.1})
to define $\nabla$, then $\nabla$ is a Weyl connection with associated 1-form $\phi$. We refer to \cite{GNU10} for further details concerning
Weyl geometry.

\subsection{Para-Hermitian manifolds}
Let $m=2\bar m$.
A triple $(M,g,J_+)$ is said to be an {\it almost para-Hermitian manifold} with an
{\it almost para-complex structure} $J_+$ if $g$ is 
a pseudo-Riemannian metric on $M$ of neutral signature $(\bar m,\bar m)$
and if $J_+$ is an endomorphism of the tangent bundle $TM$ so that $J_+^2=\operatorname{Id}$ and so that $J_+^*g=-g$;
$(M,g,J_+)$ is said to be {\it para-Hermitian} with an {\it integrable complex structure} $J_+$ if the {\it para-Nijenhuis} tensor
$$
N_{J_+}(x,y):=[x,y]-J_+ [J_+ x,y]-J_+ [x,J_+ y]+[J_+ x,J_+ y]\
$$
vanishes or, equivalently, if there are local coordinates $(u^1,...,u^{\bar m},v^1,...,v^{\bar m})$
centered at an arbitrary point of $M$ so that:
$$J_+\partial_{u_i}=\partial_{v_i}\quad\text{and}\quad J_+\partial_{v_i}=\partial_{u_i}\,.$$

\subsection{Pseudo-Hermitian manifolds}
Let $m=2\bar m$. A triple $(M,g,J_-)$ is said to be an {\it almost pseudo-Hermitian manifold} with an {\it almost complex structure}
$J_-$ if $(M,g)$ is a pseudo-Riemannian manifold, 
if $J_-$ is an endomorphism of the tangent bundle so that $J_-^2=-\operatorname{id}$ and so that $J_-^*g=g$; 
$(M,g,J_-)$ is said to be a {\it pseudo-Hermitian} manifold with an {\it integrable complex structure} $J_-$ if the {\it Nijenhuis} tensor
$$
N_{J_-}(x,y):=[x,y]+J_- [J_- x,y]+J_- [x,J_- y]-[J_- x,J_- y]
$$
vanishes or, equivalently, if there are local coordinates $(u^1,...,u^{\bar m},v^1,...,v^{\bar m})$
centered at an arbitrary point of $M$
so that:
$$J_-\partial_{u_i}=\partial_{v_i}\quad\text{and}\quad J_-\partial_{v_i}=-\partial_{u_i}\,.$$

\subsection{Para-K\"ahler and K\"ahler manifolds}
One says that a Weyl connection $\nabla$ on a para-Hermitian manifold $(M,g,J_+)$
is a {\it para-K\"ahler--Weyl} connection 
if $\nabla J_+=0$. Similarly, one says that a Weyl connection $\nabla$ on a pseudo-Hermitian manifold 
$(M,g,J_-)$ is a 
{\it K\"ahler--Weyl} connection if $\nabla J_-=0$. Since $\nabla J_\pm=0$ implies
$J_\pm$ to be integrable, we assume
this condition henceforth. If $\nabla=\nabla^g$ is the Levi-Civita connection, 
then $(M,g,J_\pm)$ is said to be {\it (para)-K\"ahler}.

Let $\star$ be the Hodge operator and let $\Omega_\pm(x,y):=g(x,J_\pm y)$
be the (para)-K\"ahler form. The co-derivative $\delta\Omega_\pm$ is given,
see \cite{G94} for example, by the formula:
$$
\delta\Omega_\pm=-\star d\star\Omega_\pm\,.
$$

The following is well known -- see,
for example, the discussion in \cite{PPS93} of the Riemannian setting (which
uses results of \cite{V82,V83}) and the generalization given in \cite{GN11}
to the more general context:

\begin{thm}\label{thm-1.1} Let $m\ge6$. If $(M,g,J_\pm,\nabla)$ is a 
(para)-K\"ahler--Weyl structure, then the associated Weyl structure is trivial, i.e.
there is always locally a conformally equivalent metric $\tilde g=e^{2f}g$ so that
$(M,\tilde g,J_\pm)$ is (para)-K\"ahler and so that $\nabla=\nabla^{\tilde g}$.
\end{thm}

By Theorem~\ref{thm-1.1}, only the $4$-dimensional setting is relevant. The following is the main result of this short note;
it plays a central role in the discussion of \cite{BGGV12}.
\begin{thm}\label{thm-1.2}
\ \begin{enumerate}
\item If $\mathcal{M}=(M,g,J_+)$ is a para-Hermitian manifold of signature $(2,2)$, then there is a unique para-K\"ahler--Weyl structure on $\mathcal{M}$
with $\phi=\frac12J_+\delta\Omega_+$.
\smallbreak\item If $\mathcal{M}=(M,g,J_-)$ is a pseudo-Hermitian manifold of signature $(2,2)$, then there is
a unique K\"ahler--Weyl structure on $\mathcal{M}$ with $\phi=-\frac12J_-\delta\Omega_-$.
\smallbreak\item If $\mathcal{M}=(M,g,J_-)$ is a Hermitian manifold of signature $(0,4)$, then there is
a unique K\"ahler--Weyl structure on $\mathcal{M}$ with $\phi=-\frac12J_-\delta\Omega_-$.
\end{enumerate}\end{thm}

Assertion (3) of Theorem~\ref{thm-1.2}, which deals with the Hermitian setting, is well known -- see, for example,
the discussion in \cite{KK10}. 
Subsequently, Theorem~\ref{thm-1.2} was established full generality (see \cite{GN11,GN12}) by extending the 
Higa curvature decomposition \cite{H93,H94} from the real to the K\"ahler--Weyl and to
the para-K\"ahler Weyl contexts. 

Here is a brief outline to this paper. In Section~\ref{sect-2}, we will show that if a (para)-K\"ahler Weyl structure exists, then it
 is unique. In Section~\ref{sect-3}, we will give a direct proof of Assertion (1) of Theorem~\ref{thm-1.2} in the
para-Hermitian setting. In Section~\ref{sect-4}, we will use analytic continuation to derive 
Assertions (2) and (3), which deal with the complex setting, from Assertion (1).
This reverses the usual procedure of viewing para-complex geometry
setting as an adjunct to complex geometry and is a novel feature of this paper.

\section{Uniqueness of the (para)-K\"ahler--Weyl structure}\label{sect-2}
This section is devoted to the proof of the following uniqueness result:

\begin{lem}
\ \begin{enumerate}
\item If $\nabla^{\phi_1}$ and $\nabla^{\phi_2}$ are two para-K\"ahler--Weyl connections on a $4$-dimensional para-Hermitian manifold $(M,g,J_+)$,
then $\phi_1=\phi_2$.
\item If $\nabla^{\phi_1}$ and $\nabla^{\phi_2}$ are two K\"ahler--Weyl connections on a $4$-dimensional pseudo-Hermitian manifold $(M,g,J_-)$,
then $\phi_1=\phi_2$.
\end{enumerate}
\end{lem}

\begin{proof}
Let $\phi=\phi_1-\phi_2$ and let
$$\Theta_X(Y)=\phi(X)Y+\phi(Y)X-g(X,Y)\phi^{\#}\,.$$
By Equation~(\ref{eqn-1.1}),  
$\nabla_X^{\phi_1}-\nabla_X^{\phi_2}=\Theta_X\in\operatorname{End}(TM)$. Consequently,
$$\{\nabla^{\phi_1}-\nabla^{\phi_2}\}J_\pm=0\quad\Rightarrow\quad [\Theta_X,J_\pm]=0\text{ for all }X\,.$$

We first deal with the para-Hermitian case. This is a purely algebraic computation. Let
$\{e_1,e_2,e_3,e_4\}$ be a local frame for $TM$ so that 
\begin{equation}\label{eqn-2.1}
\begin{array}{l}
J_+e_1=e_1,\quad J_+e_2=e_2,\quad J_+e_3=-e_3,\quad J_+e_4=-e_4,\\
g(e_1,e_3)=g(e_2,e_4)=1\,.\vphantom{\vrule height 12pt}
\end{array}\end{equation}
We expand $\phi=a_1e^1+a_2e^2+a_3e^3+a_4e^4$ and compute:
\begin{eqnarray*}
&&\Theta_{e_1}e_4=a_1e_4+a_4e_1,\ 
J_+\Theta_{e_1}e_4=-a_1e_4+a_4e_1,\ 
\Theta_{e_1}J_+e_4=-a_1e_4-a_4e_1,\\
&&\Theta_{e_2}e_3=a_2e_3+a_3e_2,\ 
J_+\Theta_{e_2}e_3=-a_2e_3+a_3e_2,\ 
\Theta_{e_2}J_+e_3=-a_2e_3-a_3e_2,\\
&&\Theta_{e_4}e_1=a_4e_1+a_1e_4,\ 
J_+\Theta_{e_4}e_1=\phantom{-}a_4e_1-a_1e_4,\ 
\Theta_{e_4}J_+e_1=\phantom{-}a_4e_1+a_1e_4,\\
&&\Theta_{e_3}e_2=a_3e_2+a_2e_3,\ 
J_+\Theta_{e_3}e_2=\phantom{-}a_3e_2-a_2e_3,\ 
\Theta_{e_3}J_+e_2=\phantom{-}a_3e_2+a_2e_3\,.
\end{eqnarray*}
Equating $\Theta_{e_i}J_+e_j$ with $J_+\Theta_{e_i}e_j$ then implies
$a_1=a_2=a_3=a_4=0$ so $\phi=0$ and $\phi_1=\phi_2$. This establishes Assertion~(1).

Next assume we are in the pseudo-Hermitian setting. Complexify and extend $g$ to be complex bilinear. Choose a local frame $\{Z_1,Z_2,\bar Z_1,\bar Z_2\}$ for
$TM\otimes_{\mathbb{R}}\mathbb{C}$ so
$$\begin{array}{ll}
J_-Z_1=\sqrt{-1}Z_1,&J_-Z_2=\sqrt{-1}Z_2,\\
J_-\bar Z_1=-\sqrt{-1}\bar Z_1,&J_-\bar Z_2=-\sqrt{-1}\bar Z_2,\\
g(Z_1,\bar Z_1)=1,&g(Z_2,\bar Z_2)=\varepsilon_2
\end{array}$$
where we take $\varepsilon_2=+1$ in signature $(0,4)$ and $\varepsilon_2=-1$ in signature $(2,2)$. We set
$J_+:=-\sqrt{-1}J_-$, $e_1:=Z_1$, $e_2:=Z_2$, $e_3:=\bar Z_1$, and $e_4:=\varepsilon_2\bar Z_2$ and
apply the argument given to prove Assertion~(1) (where the coefficients $a_i$ are now complex) to derive Assertion~(2). \end{proof}

\section{Para-Hermitian geometry}\label{sect-3}
\subsection{The algebraic context}
Let $(V,\langle\cdot,\cdot\rangle,J_+)$ be a para-Hermitian vector space of dimension $4$. 
Here $\langle\cdot,\cdot\rangle$ is an inner product 
on $V$ of signature $(2,2)$ and $J_+$ is an endomorphism of $V$ satisfying $J_+^2=\operatorname{Id}$ and 
$J_+^*\langle\cdot,\cdot\rangle=-\langle\cdot,\cdot\rangle$. We may then choose a basis $\{e_1,e_2,e_3,e_4\}$ for $V=\mathbb{R}^4$ so that
the relations of Equation~(\ref{eqn-2.1}) are satisfied. The K\"ahler form and orientation $\dvol$ are then given by:
$$\Omega_+=-e^1\wedge e^3-e^2\wedge e^4\quad\text{ and }\quad
\dvol=\textstyle\frac12\Omega_+\wedge\Omega_+=e^1\wedge e^3\wedge e^2\wedge e^4\,.$$
Let $\star$ be the {\it Hodge operator}. This operator is characterized by the relation:
$$\omega_1\wedge\star\omega_2=\langle\omega_1,\omega_2\rangle e^1\wedge e^3\wedge e^2\wedge e^4\text{ for all }\omega_i\,.
$$
Consequently:
\begin{equation}\label{eqn-3.a}
\begin{array}{ll}
\star e^1\wedge e^3=-e^2\wedge e^4,&
\star e^2\wedge e^4=-e^1\wedge e^3,\\
\star e^1\wedge e^2\wedge e^3=-e^2,&
\star e^1\wedge e^2\wedge e^4=\phantom{-}e^1,\\
\star e^1\wedge e^3\wedge e^4=-e^4,&
\star e^2\wedge e^3\wedge e^4=\phantom{-}e^3,\vphantom{\vrule height 10pt}.
\end{array}\end{equation}

\subsection{Example}\label{sect-3.2}
\rm We begin the proof of Theorem~\ref{thm-1.2} by considering a very specific example.
Let $(x^1,x^2,x^3,x^4\}$ be the usual coordinates on $\mathbb{R}^4$,
let $\partial_i:=\partial_{x_i}$, and let $J_+$ be the standard para-complex structure:
$$J_+\partial_1=\partial_1,\quad J_+\partial_2=\partial_2,\quad
    J_+\partial_3=-\partial_3,\quad J_+\partial_4=-\partial_4\,.$$
Let $f(0)=0$. We take the metric to have non-zero components determined by:
$$g(\partial_1,\partial_3)=1\text{ and }
    g(\partial_2,\partial_4)=e^{2f}\,.
$$
and let $f_i:=\{\partial_if\}(0)$. The (possibly) non-zero Christoffel symbols of $\nabla^g$ at the origin are given by:
$$\begin{array}{l}
g(\nabla^g_{\partial_1}\partial_2,\partial_4)=
g(\nabla^g_{\partial_2}\partial_1,\partial_4)=
g(\nabla^g_{\partial_1}\partial_4,\partial_2)=
g(\nabla^g_{\partial_4}\partial_1,\partial_2)=f_1,\\
g(\nabla^g_{\partial_3}\partial_2,\partial_4)=
g(\nabla^g_{\partial_2}\partial_3,\partial_4)=
g(\nabla^g_{\partial_3}\partial_4,\partial_2)=
g(\nabla^g_{\partial_4}\partial_3,\partial_2)=f_3,\\
g(\nabla^g_{\partial_4}\partial_4,\partial_2)=2f_4,\qquad\qquad\qquad\qquad
g(\nabla^g_{\partial_2}\partial_2,\partial_4)=2f_2,\\
g(\nabla^g_{\partial_2}\partial_4,\partial_1)=
g(\nabla^g_{\partial_4}\partial_2,\partial_1)=-f_1,\quad
g(\nabla^g_{\partial_2}\partial_4,\partial_3)=
g(\nabla^g_{\partial_4}\partial_2,\partial_3)=-f_3,\\
\end{array}$$
Consequently the (possibly) non-zero covariant derivatives at the origin are:
$$\begin{array}{ll}
\nabla^g_{\partial_1}\partial_2=\nabla^g_{\partial_2}\partial_1=f_1\partial_2,&
  \nabla^g_{\partial_1}\partial_4=\nabla^g_{\partial_4}\partial_1=f_1\partial_4,\\
\nabla^g_{\partial_3}\partial_2=\nabla^g_{\partial_2}\partial_3=f_3\partial_2,&
  \nabla^g_{\partial_3}\partial_4=\nabla^g_{\partial_4}\partial_3=f_3\partial_4,\\
\nabla^g_{\partial_4}\partial_4=2f_4\partial_4,&
\nabla^g_{\partial_2}\partial_2=2f_2\partial_2,\\
\nabla^g_{\partial_2}\partial_4=\nabla^g_{\partial_4}\partial_2
=-f_1\partial_3-f_3\partial_1.
\end{array}$$
Since $\nabla^g_{\partial_1}$ and $\nabla^g_{\partial_3}$ are diagonal, they commute with $J_+$ so
$\nabla^g_{\partial_1}(J_+)=\nabla^g_{\partial_3}(J_+)=0$. We compute:
\begin{eqnarray*}
&&(\nabla^g_{\partial_2}J_+)\partial_1=(1-J_+)\nabla^g_{\partial_2}\partial_1
=(1-J_+)f_1\partial_2=0,\\
&&(\nabla^g_{\partial_2}J_+)\partial_2=(1-J_+)\nabla^g_{\partial_2}\partial_2=(1-J_+)2f_2\partial_2=0,\\
&&(\nabla^g_{\partial_2}J_+)\partial_3=(-1-J_+)\nabla^g_{\partial_2}\partial_3=
(-1-J_+)f_3\partial_2=-2f_3\partial_2,\\
&&(\nabla^g_{\partial_2}J_+)\partial_4=(-1-J_+)\nabla^g_{\partial_2}\partial_4=
(-1-J_+)(-f_1\partial_3-f_3\partial_1)
=2f_3\partial_1,\\
&&(\nabla^g_{\partial_4}J_+)\partial_1=(1-J_+)\nabla^g_{\partial_4}\partial_1
=(1-J_+)f_1\partial_4=2f_1\partial_4,\\
&&(\nabla^g_{\partial_4}J_+)\partial_2=(1-J_+)\nabla^g_{\partial_4}\partial_2
=(1-J_+)(-f_1\partial_3-f_3\partial_1)=-2f_1\partial_3,\\
&&(\nabla^g_{\partial_4}J_+)\partial_3=(-1-J_+)\nabla^g_{\partial_4}\partial_3=
(-1-J_+)f_3\partial_4=0,\\
&&(\nabla^g_{\partial_4}J_+)\partial_4=(-1-J_+)\nabla^g_{\partial_4}\partial_4=(-1-J_+)2f_4\partial_4=0\,.
\end{eqnarray*}
We apply Equation~(\ref{eqn-3.a}). We have $\star\Omega_+=-\Omega_+$.
Setting $e^1=dx^1$, $e^2=e^fdx^2$, $e^3=dx^3$, and $e^4=e^fdx^4$ and recalling $f(0)=0$ yields
\begin{eqnarray*}
&&\star \Omega_+=-\Omega_+=dx^1\wedge dx^3+e^{2f}dx^2\wedge dx^4,\\
&&d\star \Omega_+=2f_1dx^1\wedge dx^2\wedge dx^4-2f_3dx^2\wedge dx^3\wedge dx^4,\\
&&\delta_g\Omega_+(0)=-\star d\star \Omega_+(0)=-2f_1dx^1+2f_3dx^3,\\
&&\phi(0)=\textstyle\frac12J\delta_g\Omega_+=-f_1dx^1-f_3dx^3,\text{ and }
\phi^{\#}=-f_1\partial_3-f_3\partial_1\,.
\end{eqnarray*}
Let
$\Theta_{ij}:=\phi(\partial_i)\partial_j+\phi(\partial_j)\partial_i
-g(\partial_i,\partial_j)\phi^{\#}=(\nabla^\phi-\nabla^g)_{\partial_i}\partial_j$ at $0$. Then:
\begin{eqnarray*}
&&\Theta_{11}=-2f_1\partial_1,\quad\Theta_{12}=-f_1\partial_2,\\
&&\Theta_{13}=(-f_1\partial_3-f_3\partial_1)+(f_1\partial_3+f_{3}\partial_1)=0,\\
&&\Theta_{14}=-f_1\partial_4,\quad\Theta_{22}=0,\quad\Theta_{23}=-f_3\partial_2,\\
&&\Theta_{24}=(f_1\partial_3+f_3\partial_1),\quad
\Theta_{33}=-2f_3\partial_3,\quad\Theta_{34}=-f_3\partial_4,
\quad\Theta_{44}=0\,.
\end{eqnarray*}
Since $\Theta(\partial_1)$ and $\Theta(\partial_3)$ are diagonal,
 $[\Theta(\partial_1),J_+]=[\Theta(\partial_3),J_+]=0$. We compute:
 \begin{eqnarray*}
 &&[\Theta(\partial_2),J_+]\partial_1=(1-J_+)\Theta_{12}=0,\\
 &&[\Theta(\partial_2),J_+]\partial_2=(1-J_+)\Theta_{22}=0,\\
 &&[\Theta(\partial_2),J_+]\partial_3=(-1-J_+)\Theta_{23}=2f_3\partial_2,\\
 &&[\Theta(\partial_2),J_+]\partial_4=(-1-J_+)\Theta_{24}=-2f_3\partial_1,\\
 &&[\Theta(\partial_4),J_+]\partial_1=(1-J_+)\Theta_{14}=-2f_1\partial_4,\\
 &&[\Theta(\partial_4),J_+]\partial_2=(1-J_+)\Theta_{24}=2f_1\partial_3,\\
 &&[\Theta(\partial_4),J_+]\partial_3=(-1-J_+)\Theta_{34}=0,\\
 &&[\Theta(\partial_4),J_+]\partial_4=(-1-J)\Theta_{44}=0\,.
 \end{eqnarray*}
 We now observe that $[\nabla^g,J_+]+[\Theta,J_+]=0$. Consequently $\nabla^\phi J_+=0$ for this
metric and Assertion~(1) of Theorem~\ref{thm-1.2} holds in this special case.

 \subsection{The proof of Theorem~\ref{thm-1.2}~(1)}
  Let $V=\mathbb{R}^4$, let $S^2_-$ be the vector space of symmetric $2$-cotensors $\omega$ so that $J_+^*\omega=-\omega$,
   and let
 $\varepsilon\in C^\infty(S^2)$
 satisfy $\varepsilon(0)=0$. We use $\varepsilon$ to define a perturbation of the flat metric by setting:
 $$g=dx^1\circ dx^3+dx^2\circ dx^4+\varepsilon\,.$$
 This is non-degenerate near the origin. Since only the $1$-jets of $\varepsilon$ are relevant in examining $\nabla^\phi(J_+)(0)$, 
 this is a linear problem and we may take $\varepsilon\in S_-^2\otimes V^\star$ so:
 $$g=g_0+\sum_ix^i\varepsilon(e_i)\,.$$
Then $\varepsilon\rightarrow(\nabla^\phi J_+)(0)$ defines a linear map
\begin{eqnarray*}
&&\mathcal{E}:S_-(V)\otimes V^*\rightarrow\operatorname{End}(V)\otimes V^*
\text{ or equivalently}\\
&&\mathcal{E}:S_-(V)\rightarrow\operatorname{Hom}(V^*,\operatorname{End}(V)\otimes V^*)\,.
\end{eqnarray*}
The analysis of Section~\ref{sect-3.2} shows that $\mathcal{E}(dx^2\circ dx^4)=0$.
Permuting the indices $1\leftrightarrow2$ and $3\leftrightarrow4$ then yields
$\mathcal{E}(dx^1\circ dx^3)=0$. The question is
invariant under the action of the para-unitary group; we must preserve $J_+$ and we must
preserve the inner product at the origin. Define a unitary transformation $T$ by setting:
$$\begin{array}{ll}
T(e^1)=e^1+ae^2,&T(e^2)=e^2,\\
T(e^3)=e^3,&T(e^4)=e^4-ae^3.
\end{array}$$
Then
$$T(e^1\wedge e^3)=e^1\wedge e^3+ae^2\wedge e^3$$
Consequently, $\mathcal{E}(e^2\wedge e^3)=0$. Permuting the indices
$1\leftrightarrow2$ and $3\leftrightarrow4$
then yields $\mathcal{E}(e^1\wedge e^4)=0$. Since 
$$
\mathcal{S}_-=\operatorname{Span}\{e^1\wedge e^3,e^1\wedge e^4,e^2\wedge e^3,e^2\wedge e^4\}\,,
$$
we see that $\mathcal{E}=0$ in general; this completes
the proof of Assertion~(1) of Theorem~\ref{thm-1.2}.\hfill\qed

\section{Hermitian and pseudo-Hermitian manifolds}\label{sect-4}

In Section~\ref{sect-4}, we will use analytic continuation to derive
Theorem~\ref{thm-1.2} in the complex setting from Theorem~\ref{thm-1.2} in the para-complex setting.
Let $V=\mathbb{R}^4$ with the usual basis $\{e_1,e_2,e_3,e_4\}$ 
and coordinates
$\{x^1,x^2,x^3,x^4\}$ where we expand $v=x^1e_1+x^2e_2+x^3e_3+x^4e_4$. Let $S^2$ denote the space of symmetric $2$-tensors. We
complexity and consider 
$$
\mathcal{S}:=\left\{S^2\otimes_{\mathbb{R}}\mathbb{C}\right\}\ \oplus\ 
\left\{(V^*\otimes_{\mathbb{R}}S^2)\otimes_{\mathbb{R}}\mathbb{C}\right\}\,.
$$
Let $J_+\in M_2(\mathbb{C})$ be a complex $2\times 2$ matrix with $J_+^2=\operatorname{Id}$ and $\operatorname{Tr}(J_+)=0$.
Let:
\begin{equation}\label{eqn-4.a}
\mathcal{S}(J_+):=\{(g_0,g_1)\in\mathcal{S}:\det(g_0-J_+^*g_0)\ne0\}\,.
\end{equation}
For $(g_0,g_1)\in\mathcal{S}(J_+)$, 
define:
\begin{eqnarray*}
&&g(x)(X,Y):=\textstyle\frac12\{g_0(X,Y)-g_0(J_+X,J_+Y)\}\\
&&\qquad+\sum_{i=1}^4x^i\cdot{\textstyle\frac12}\left\{g_1(e_i,X,Y)-g_1(e_i,J_+X,J_+Y)\right\}\,.
\end{eqnarray*}
By Equation~(\ref{eqn-4.a}), this is non-degenerate at $0$ and defines a complex metric on some neighborhood of $0$ so $J_+^*g=-g$. 
Let $\nabla^g$ be the complex Levi--Civita connection:
$$\nabla^g_{\partial_i}\partial_j=\textstyle\frac12g^{kl}\{\partial_ig_{jl}+\partial_jg_{il}-\partial_{x_l}g_{ij}\}\partial_k\,.$$
Then $\nabla^g$ is a torsion free connection on 
$T_{\mathbb{C}}M:=T_M\otimes_{\mathbb{R}}\mathbb{C}$. The para-K\"ahler form
is defined by setting $\Omega_+(x,y)=g(x,J_+y)$ and
we have
$$\delta\Omega_+=\star d\Omega_+\text{ and }\phi:=\textstyle\frac12J_+\delta_g\Omega\,.
$$
We then use $\phi$ to define a complex Weyl connection $\nabla^\phi$ on $T_{\mathbb{C}}M$ and 
define a holomorphic map from $\mathcal{S}(J_+)$ to 
$\mathfrak{V}:=V^*\otimes M_4(\mathbb{C})$
by setting:
$$\mathcal{E}(g_0,g_1;J_+):=\nabla^\phi(J_+)|_{x=0}\,.$$
\begin{lem}\label{lem-4.1}
Let $J_+\in M_4(\mathbb{C})$ with $J_+^2=\operatorname{id}$ and $\operatorname{Tr}(J_+)=0$. 
Suppose that $(g_0,g_1)\in\mathcal{S}(J_+)$.
\begin{enumerate}
\item If $J_+$ is real and if $(g_0,g_1)$ is real, then $\mathcal{E}(g_0,g_1;J_+)=0$.
\item If $J_+$ is real and if $(g_0,g_1)$ is complex, then $\mathcal{E}(g_0,g_1;J_+)=0$.
\item If $J_+$ is complex and if $(g_0,g_1)$ is complex, then $\mathcal{E}(g_0,g_1;J_+)=0$.
\end{enumerate}
\end{lem}

\begin{proof} Assertion (1) follows from Theorem~\ref{thm-1.2}~(1).
We argue as follows to prove Assertion~(2).
$\mathcal{S}(J_+)$ is an open dense subset of $\mathcal{S}$ and inherits a natural
holomorphic structure thereby. Assume that $J_+$ is real. The map $\mathcal{E}$ is a holomorphic
map from $\mathcal{S}(J_+)$ to $\mathfrak{V}$. By
Assertion~(1), $\mathcal{E}(g_0,g_1;J_+)$ vanishes if $(g_0,g_1)$ is real. Thus, by the identity theorem, $\mathcal{E}(g_0,g_1;J_+)$ vanishes
for all $(g_0,g_1)\in{\mathcal{S}}_{J_+}$. This establishes Assertion~(2) by removing the assumption that $(g_0,g_1)$ is real.

We complete the proof by removing the assumption that $J_+$ is real.
The general linear group $\operatorname{GL}_4(\mathbb{C})$ acts on the structures involved
by change of basis (i.e. conjugation). Let $(g_0,g_1)\in\mathcal{S}(J_+)$ where $J_+$ is real and $\operatorname{Tr}(J_+)=0$. 
We consider the real
and complex orbits:
\begin{eqnarray*}
&&\mathcal{O}_{\mathbb{R}}(g_0,g_1;J_+):=\operatorname{GL}_4(\mathbb{R})\cdot(g_0,g_1;J_+),\\
&&\mathcal{O}_{\mathbb{C}}(g_0,g_1;J_+):=\operatorname{GL}_4(\mathbb{C})\cdot(g_0,g_1;J_+)\,.
\end{eqnarray*}
Let $\mathcal{F}(A):=\mathcal{E}(A\cdot(g_0,g_1;J_+))$ define a holomorphic map
from $\operatorname{GL}_4(\mathbb{C})$ to $\mathfrak{V}$. 
By Assertion (2), $\mathcal{F}$ vanishes on $\operatorname{GL}_4(\mathbb{R})$. Thus by
the identity theorem, $\mathcal{F}$ vanishes on $\operatorname{GL}_4(\mathbb{C})$ or,
equivalently, $\mathcal{E}$ vanishes on the orbit space $\mathcal{O}_{\mathbb{C}}(g_0,g_1;J_+)$.
Given any $J_+\in M_4(\mathbb{C})$ with $J_+^2=\operatorname{Id}$ and $\operatorname{Tr}(J_+)=0$, we
can choose $A\in\operatorname{GL}_4(\mathbb{C})$ so that $A\cdot J_+$ is real.
The general case now follows from Assertion (2).
\end{proof}

\subsection{The proof of Theorem~\ref{thm-1.2}~(2,3)} Let $(M,g,J_-)$ be a $4$-dimensional pseudo-Hermitian
manifold of dimension $4$. Fix a point $P$ of $M$. Since $J_-$ is integrable, we may
choose local coordinates
$(x^1,x^2,x^3,x^4)$ so the matrix of $J_-$ relative to the coordinate frame $\{\partial_i\}$
is constant. Define a Weyl connection with associated 1-form given by
$\phi=-\frac12J_-\delta\Omega_-$. Only the $0$ and the $1$-jets of
the metric play a role in the computation of $(\nabla^\phi J_-)(P)$. 
So we may assume $g=g(g_0,g_1)$. We set $J_+=\sqrt{-1}J_-$.
We have that
\begin{eqnarray*}
&&J_+^2=\sqrt{-1}J_-\sqrt{-1}J_-=-J_-^2=\operatorname{id},\quad
  \operatorname{Tr}(J_+)=\sqrt{-1}\operatorname{Tr}(J_-)=0,\\
 &&J_+^\star(g)(X,Y)=g(\sqrt{-1}J_-X,\sqrt{-1}J_-Y)=-g(J_-X,J_-Y)=-g(X,Y)
 \end{eqnarray*}
so $J_+^\star(g)=-g$ and $(g_0,g_1)\in\mathcal{S}_{J_+}$. Finally, since $J_-=-\sqrt{-1}J_+$, we have
\begin{eqnarray*}
&&\Omega_-=-\sqrt{-1}\Omega_+,\\
&&\phi_{J_-}=-\textstyle\frac12J_-\delta_g\Omega_-=-\textstyle\frac12(-\sqrt{-1}J_+)\delta_g(-\sqrt{-1}\Omega_+)=
\frac12J_+\delta_g\Omega_+=\phi_{J_+}\,.
\end{eqnarray*}
We apply Lemma~\ref{lem-4.1} to complete the proof.
\hfill\qed

\end{document}